\documentclass{gtpart}   
\usepackage{pinlabel}
\usepackage[all]{xy}

\title[On orbifold Riemannian metrics]{On real analytic orbifolds and  Riemannian metrics}

\author{Marja Kankaanrinta} 
\givenname{Marja}
\surname{Kankaanrinta}
\address{Department of Mathematics\\University of Virginia \\\newline
         Charlottesville, VA 22903\\USA}
\email{mk5aq@virginia.edu}



\newtheorem{theorem}{Theorem}[section]   
\newtheorem{lemma}[theorem]{Lemma}
\newtheorem{prop}[theorem]{Proposition}         
 \newtheorem{cor}[theorem]{Corollary}
\theoremstyle{definition}
\newtheorem{definition}[theorem]{Definition}   
\newtheorem*{remark}{Remark}

\makeop{Homo}
\numberwithin{equation}{section}

\begin{document}

%
%
%
%
%
%
%
%
\begin{abstract}  
We begin by showing  that every real analytic orbifold has a real analytic
Riemannian metric. It follows that every reduced real analytic orbifold
can be expressed as a quotient of a real analytic manifold by a real
analytic almost free action of a compact Lie group. We then extend a
well-known result of Nomizu and Ozeki concerning
 Riemannian metrics on manifolds
to the orbifold setting:  Let $X$ be a smooth (real analytic) orbifold and let
$\alpha$ be a smooth (real analytic) Riemannian metric on $X$. Then $X$
has a complete smooth (real analytic) Riemannian metric conformal to
$\alpha$.
\end{abstract}

\maketitle
\section{Introduction}
\label{intro}
\noindent    In this paper we consider Riemannian metrics
on smooth, i.e., ${\rm C}^\infty$, and real analytic orbifolds.  As is well-known,
a smooth Riemannian metric for any smooth orbifold can be
constructed by using invariant Riemannian metrics on orbifold charts and
gluing them together by a smooth partition of unity. Real analytic manifolds
admit real analytic Riemannian metrics, since they can be real analytically
embedded in Euclidean spaces. Neither of these two methods 
to construct Riemannian metrics work for real
analytic orbifolds and a different approach is needed.

Recall that an orbifold is called {\it reduced} if the actions of the finite groups
on  orbifold charts are effective.
We first study the frame bundle ${\rm Fr}(X)$ of a reduced
$n$-dimensional real analytic  
orbifold $X$.  The frame bundle ${\rm Fr}(X)$ is a real analytic manifold
and the general linear group ${\rm GL}_n({\mathbb{R}})$ acts properly 
and almost freely, i.e., with finite isotropy subgroups, on
${\rm Fr}(X)$. Thus ${\rm Fr}(X)$ has a ${\rm GL}_n({\mathbb{R}})$-invariant
real analytic Riemannian metric
(\cite{IK}, Theorem I), which induces a real analytic Riemannian
metric on $X$. We then show that every real analytic orbifold inherits a
real analytic Riemannian metric from  the corresponding reduced orbifold.
Therefore we obtain:

\begin{theorem}
\label{realanalcase}
Let $X$ be a real analytic orbifold. Then $X$ has a real analytic Riemannian
metric.
\end{theorem}

Since, by Theorem \ref{realanalcase},
every real analytic orbifold has a real analytic Riemannian metric, we can
construct the orthonormal frame bundle ${\rm OFr}(X)$ for every reduced
real analytic orbifold. Exactly as in the smooth case 
(\cite{ALR}, Theorem 1.23) we prove:

\begin{theorem}
\label{Orthonframe}
Let $X$ be a reduced $n$-dimensional real analytic orbifold. Then $X$ is real analytically
diffeomorphic to the quotient orbifold ${\rm OFr}(X)/{{\rm O}(n)}$.
\end{theorem}

Notice that if $X$ is a reduced $n$-dimensional
real analytic orbifold, then ${\rm OFr}(X)$ 
is a real analytic manifold with a real analytic, effective,
almost free action of the orthogonal group 
${\rm O}(n)$.
Thus Theorem \ref{Orthonframe} implies the following:

\begin{cor}
\label{quotbygroup}
Let $X$ be a reduced $n$-dimensional real analytic orbifold. Then
$X$ is real analytically diffeomorphic to a quotient orbifold $M/{\rm O}(n)$,
where $M$ is a real analytic manifold and ${\rm O}(n)$ acts on $M$
real analytically, effectively  and almost freely.
\end{cor}

It follows that reduced real analytic orbifolds can be studied by using  methods 
developed for studying real analytic almost free actions of compact Lie groups.

To prove Theorems \ref{realanalcase} and \ref{Orthonframe}, we use two
kinds of comparisons. Firstly,  we compare Riemannian metrics
on a quotient orbifold $M/G$ to $G$-invariant Riemannian metrics on
the $G$-manifold $M$ (Section \ref{quotmetr}).  Secondly, we
compare Riemannian metrics on an orbifold  to those on  the 
corresponding reduced orbifold (Section \ref{compX}). 
We conclude the paper  by applying these comparisons to prove
a result concerning complete Riemannian metrics:

\begin{theorem}
\label{complthm}
Let $X$ be a smooth (resp. real analytic) orbifold. For any smooth 
(resp. real analytic) Riemannian metric
$\alpha$ on $X$ there exists a complete smooth 
(resp. real analytic) Riemannian metric 
on $X$ which is conformal to $\alpha$.
\end{theorem}

The corresponding result for Riemannian metrics on smooth
manifolds has been proven by K. Nomizu and H. Ozeki 
(\cite{NO}, Theorem 1).  The corresponding equivariant result,
which also is used in the proof of Theorem \ref{complthm},
was proved by the author (\cite{Ka},  Theorems 3.1 and 5.2).

{{\it Acknowledgements.} 
The author's research was supported by the visitors program of the
Aalto University during the academic year 2012 - 2013. The author would
like to thank the Department of  Mathematics and Systems Analysis 
of the Aalto University for its hospitality during her stay.}

\section{Definitions}
\label{orbi}

\noindent We first recall the definition of an orbifold:

\begin{definition}
\label{orbichartdef}
Let $X$ be a topological space and let $n\in {\mathbb{N}}$.
{\begin{enumerate}
\item An {\it orbifold chart} of $X$   is a triple $(\widetilde{U}, G, \varphi)$,
where $\widetilde{U}$ is an open connected subset of ${\mathbb{R}}^n$,
$G$ is a finite group acting  on $\widetilde{U}$ and
$\varphi\colon \widetilde{U}\to X$ is a $G$-invariant map inducing
a homeomorphism 
$U=\varphi(\widetilde{U})\cong\widetilde{U}/G$. Let ${\rm ker}(G)$ be the subgroup
of $G$ acting trivially on $\widetilde{U}$.
\item An {\it embedding} $(\lambda,\theta)\colon (\widetilde{U}, G, \varphi)\to
(\widetilde{V}, H,\psi)$ between two orbifold charts is  an injective
homomorphism $\theta\colon G\to H$ such that $\theta$ is an
isomorphism from ${\rm ker}(G)$ to ${\rm ker}(H)$, and
an  equivariant embedding
$\lambda\colon \widetilde{U}\to\widetilde{V}$ with $\psi\circ\lambda=\varphi$.
\item An {\it orbifold atlas} on $X$ is a family ${\cal{U}}=\{ (\widetilde{U}, G,\varphi)\}$
of orbifold charts  which cover $X$ and satisfy the following: For any two
charts $(\widetilde{U},G,\varphi)$ and $(\widetilde{V}, H,\psi)$ and a point $x\in
\varphi(\widetilde{U})\cap \psi(\widetilde{V})$, there exist a chart $(\widetilde{W}, K,\mu)$ 
such that $x\in\mu(\widetilde{W})$ and
embeddings $(\widetilde{W}, K,\mu)\to (\widetilde{U},G,\varphi)$ and
$(\widetilde{W}, K,\mu)\to (\widetilde{V}, H,\psi)$.
\item An orbifold atlas ${\cal{U}}$ {\it refines} another orbifold atlas
${\cal{V}}$ if every chart in ${\cal{U}}$ can be embedded into some
chart in ${\cal{V}}$. Two orbifold atlases are {\it equivalent} if they have a common
refinement.
\end{enumerate}}
\end{definition}

\begin{definition}
\label{orbidef}
An $n$-dimensional  {\it orbifold} is a paracompact Hausdorff space
$X$ equipped with an equivalence class of 
$n$-dimensional orbifold atlases.
\end{definition}

An orbifold is called {\it smooth} (resp. {\it real analytic}), if 
for every orbifold chart $(\widetilde{U}, G,\varphi)$, $G$ 
acts smoothly (resp. real analytically) on $\widetilde{U}$ and if each embedding
$\lambda\colon\widetilde{U}\to\widetilde{V}$ is smooth (resp. real analytic).

Let $X$ be an orbifold, and let $x\in X$. Let
$(\widetilde{U}, G,\varphi)$  and $(\widetilde{V}, H,\psi)$
be orbifold charts of  $X$ such that 
$x\in\varphi(\widetilde{U})\cap \psi(\widetilde{V})$. 
Let $\tilde{x}\in\widetilde{U}$ and $\tilde{y}\in\widetilde{V}$
be such that $\varphi(\tilde{x})=\psi(\tilde{y})=x$. 
We denote the isotropy subgroups at $\tilde{x}$ and $\tilde{y}$
by $G_{\tilde{x}}$ and $H_{\tilde{y}}$, respectively. Then
$G_{\tilde{x}}$ and $H_{\tilde{y}}$ are isomorphic. Thus 
we can associate to every $x\in X$ a finite group, well-defined
up to an isomorphism, and called {\it the local group} of $x$.

Orbifold maps are defined as follows:

\begin{definition}
\label{orbimap}
Let $X$ and $Y$ be smooth (real analytic) orbifolds.  We call a map
$f\colon X\to Y$ a {\it smooth (real analytic) orbifold map}, 
if for every $x\in X$, there
are charts $(\widetilde{U}, G,\varphi)$ around $x$ and $(\widetilde{V}, H,\psi)$
around $f(x)$, such that $f$ maps $U=\varphi(\widetilde{U})$ into
$V=\psi(\widetilde{V})$ and the restriction $f\vert U$
can be lifted to a smooth
(real analytic)  equivariant map
$\widetilde{f}\colon\widetilde{U}\to \widetilde{V}$. A smooth (real analytic)
map $f\colon X\to Y$ is called  a {\it smooth (real analytic)
diffeomorphism} if $f$ is a bijection and if the inverse map
$f^{-1}\colon Y\to X$ is smooth (real analytic).
\end{definition}

\begin{remark}
\label{huomatutus}
Let $X$ be an orbifold and let $F$ be any finite group.
Replace every orbifold chart $(\widetilde{U}, G,\varphi)$ of $X$ by the orbifold chart
$(\widetilde{U}, G\times F, \varphi)$, where $G\times F$ acts on $\widetilde{U}$
via the projection $G\times F\to G$. Doing this for every orbifold chart
of $X$ yields an orbifold $Y$. The identity maps $X\to Y$
and $Y\to X$ are orbifold maps, and they are smooth (real analytic) if
$X$ and $Y$ are smooth (real analytic) orbifolds. This means that
two orbifolds are not necessarily considered equivalent, even if they
are diffeomorphic.  However, if there is a diffeomorphism $f\colon X\to Y$,
where both $X$ and $Y$ are reduced orbifolds,
then $X$ and $Y$ have equivalent sheaf categories (\cite{MP}, Proposition 2.1).
In particular, in this case  the local groups  of $x$ and
$f(x)$ are isomorphic, for every $x\in X$.
\end{remark}





\begin{definition}
\label{riemmetriikka}
A {\it Riemannian metric} $\alpha$
on an orbifold $X$ is given by a collection of Riemannian metrics
$\alpha^{\widetilde{U}}$on the  $\widetilde{U}$ of the orbifold 
charts $(\widetilde{U}, G,\varphi)$ so that
{\begin{enumerate}
\item the group $G$ acts isometrically on $\widetilde{U}$ and
\item the embeddings  $\widetilde{W}\to \widetilde{U}$ and $\widetilde{W}\to \widetilde{V}$
of Part 3 of Definition \ref{orbichartdef} are isometries.
\end{enumerate}}
If  $X$ is a smooth (real analytic) orbifold and if
all the $\alpha^{\widetilde{U}}$ are smooth (real analytic), then $\alpha$ is
a smooth (real analytic) Riemannian metric.
\end{definition}

Let  $X$ be a smooth orbifold, and
let  $(\widetilde{U}_i, G_i,
\varphi_i)$, $i\in I$, be orbifold  charts of $X$ such that
$\{\varphi_i(\widetilde{U}_i)\}_{i\in I}$ is a locally finite cover
of $X$.  Then each $\widetilde{U}_i$ has a
smooth  Riemannian metric 
$\alpha^{\widetilde{U}_i}$, and by averaging over  $G_i$, we may assume
that  $\alpha^{\widetilde{U}_i}$ is $G_i$-invariant, i.e., that
$G_i$ acts isometrically on $\widetilde{U}_i$.
Gluing these Riemannian metrics together, by
using a smooth partition of unity,  gives
a smooth Riemannian metric on $X$
(\cite{MM}, Proposition 2.20).
All the orbifolds in \cite{MM}
are assumed to be reduced. However, the proof of Proposition
2.20 also works in the general case.

We next recall the way to define {\it distance} on a 
connected Riemannian orbifold,
for details and proofs, see \cite{Bo}. 
Assume a smooth (real analytic) orbifold $X$ is equipped with a
smooth (real analytic) Riemannian metric  $\alpha$.
Let $\gamma\colon [0,1]\to X$ be an { \it admissible}
curve (\cite{Bo}, Definition 35). 
The interval $[0,1]$ can be decomposed into finitely
many subintervals $[t_i, t_{i+1}]$ such that $\gamma( [t_i, t_{i+1}])\subset U_i=\varphi(
\widetilde{U}_i)$, for some chart $(\widetilde{U}_i, G_i,\varphi_i)$ of $X$. Let
$\gamma_i$ denote the restriction of $\gamma$
to $[t_i,t_{i+1}]$, and 
let $\widetilde{\gamma}_i$ be a lift of $\gamma_i$,  for every $i$.
If $\widetilde{\gamma}_i$ is piecewise differentiable, its length
can be calculated by integrating.
If $\widetilde{\gamma}_i$ is merely continuous, then its length can be
calculated by approximating it by piecewise differentiable  curves.
Every lift of $\gamma_i$ has the same length and
the length of the lift does not depend on which chart of $X$ is being used.
Thus the lenght  $L_\alpha(\gamma_i)$ of $\gamma_i$  can be defined 
to  be the lenght of $\widetilde{\gamma}_i$. 
Then the length $L_\alpha(\gamma)$ of $\gamma$
equals the sum of the $L_\alpha(\gamma_i)$.

Every  curve connecting two points on an orbifold can always be
replaced by an admissible curve whose local lifts are at most as long as the
ones of the original curve (\cite{Bo}, Remark 39). Thus
the distance
between any given points $x$ and $y$ of a connected orbifold
$X$ can be defined to be
$$
d_\alpha(x,y)=
\inf \{ L_\alpha(\gamma)\mid \gamma \,\,  {\rm is \,\, an\,\,  admissible\,\,
curve\,\, joining}\,\, x \,\,{\rm to}\,\,  y\}.
$$

Then $X$ equipped with the metric $d_\alpha$ becomes a metric space. 
If $d_\alpha$ is a complete metric, then any two points on 
 $X$ can be joined 
by a minimal geodesic realizing the distance $d_\alpha(x,y)$ (\cite{Bo}, Theorem
40). Moreover, $X$ is a locally compact length space. Thus it follows that
$d_\alpha$ is a complete metric if and only if the metric balls in $X$
are relatively compact.

\section{Riemannian metric on a quotient orbifold}
\label{quotmetr}

\noindent Let $G$ be a Lie group and let $M$ be a smooth
(real analytic) manifold. Assume $G$ acts on $M$ by a
smooth (real analytic) almost free action. Assume also
that the action is {\it proper}, i.e., that the map
$$
G\times M\to M\times M,\,\,\, (g,x)\mapsto (gx,x),
$$
is proper.
It is well-known that the quotient $M/G$ is a smooth
(real analytic) orbifold and that every smooth (real analytic)
$G$-invariant Riemannian metric on $M$ induces a smooth
(real analytic) Riemannian metric on $M/G$. 
We present  a proof of this basic result here (Theorem \ref{indmetr}), 
since we failed to find one in the literature.

The main idea is to use the {\it differentiable slice theorem}:
Let $x\in M$ and let $Gx$ denote the orbit of $x$.  Let
$G_x$ denote the isotropy subgroup of $G$ at $x$.
A $G$-invariant neighbourhood of $x$ can be equipped with 
a smooth (real analytic) $G$-invariant Riemannian metric.
Then there is a
$G_x$-invariant smooth (real analytic) submanifold ${\rm N}_x$
of $M$ that contains $x$ and is $G_x$-equivariantly diffeomorphic to
an open $G_x$-invariant neighbourhood of the origin in
the normal space ${\rm T}_x(M)/{\rm T}_x(Gx)$ to $Gx$ at $x$. 
The manifold ${\rm N}_x$ is called
a {\it linear slice} at $x$. It intersects the orbit $Gx$ orthogonally,
and it intersects
every orbit it meets trasversely
but not necessarily orthogonally.
The exponential map
takes an open neighbourhood of the zero section of the normal
bundle of $Gx$ diffeomorphically to the neighbourhood $G{\rm N}_x$ of $Gx$
which can be identified with the twisted product $G\times_{G_x}
{\rm N}_x$. The map 
$$
f\colon G{\rm N}_x\cong G\times_{G_x}{\rm N}_x\to G/G_x\cong Gx,\,\,\,
gy\mapsto gx,
$$
is smooth (real analytic) and $G$-equivariant. The map $f$ is exactly
the map that assigns to every point $z$ in $G{\rm N}_x$ the unique 
nearest point $f(z)$ in
$Gx$.  Thus, if $g\in G$ and
$y\in {\rm N}_x$, then the distance from $gy$ to $Gx$ equals
$d(gy,gx)=d(y,x)$, 
where $d$ denotes the metric induced by the local Riemannian metric
on the connected components of $G{\rm N}_x$.

Let $x\in M$ and let ${\rm N}_x$ be a linear slice at $x$
constructed by using a local $G$-invariant Riemannian metric $\alpha_1$.
There is a real analytic local cross section
$\delta\colon U\to G$ of the map $G\to G/G_x$, $g\mapsto gG_x$,
defined in some $G_x$-invariant
neighbourhood $U$ of $eG_x$ in $G/G_x$  and
having the property $\delta(eG_x)=e$.  We may choose $\delta$ to
be $G_x$-equivariant, i.e., $\delta(hu)=h\delta(u)h^{-1}$,
for every $h\in G_x$ and for every $u\in U$.
Let $f_0\colon G{\rm N}_x\to G/G_x$, $gy\mapsto gG_x$.
The map
$F\colon U\times {\rm N}_{x}\to V$,
$(u,s)\mapsto \delta(u)s$, is a smooth (real analytic) diffeomorphism onto some
neighbourhood $V$ of ${\rm N}_x$. The inverse of $F$ is given by
$F^{-1}\colon V\to U\times {\rm N}_{x}$,
$y\mapsto (f_0(y),\delta(f_0(y))^{-1}y)$. 
Let then $y\in {\rm N}_x$ and let
${\rm N}'_y$ be a linear slice at $y$ constructed by using a local
$G$-invariant Riemannian metric $\alpha_2$. We may assume
that ${\rm N}'_y\subset V$. 
Let ${\rm pr}\colon U\times {\rm N}_x\to
{\rm N}_x$ be the projection, and let $\lambda={\rm pr}\circ 
F^{-1}\vert\colon {\rm N}'_y\to {\rm N}_x$.  
Then $\lambda$ is an equivariant 
embedding and it
induces the identity map on the orbit space level.

It follows that the quotient $M/G$ is an orbifold with orbifold
charts $({\rm N}_x, G_x, \pi_x)$, where $x\in M$ and $\pi_x$ denotes
the  natural projection ${\rm N}_x\to {\rm N}_x/G_x
\cong (G{\rm N}_x)/G$. The ${\rm N}_x$ are defined by using local
$G$-invariant Riemannian metrics.


Let us next consider a smooth (real analytic) $G$-invariant
Riemannian metric $\alpha$ on $M$. For every $x\in M$, let
${\rm N}_x$ be a linear slice constructed by using $\alpha$.
Then  $\alpha$ induces a smooth
(real analytic) $G_x$-invariant Riemannian metric $\alpha\vert{\rm N}_x$ on
${\rm N}_x$, for every $x$.  (The inner product on ${\rm T}_z{\rm N}_x$,
$z\in {\rm N}_x$, is
given by first projecting to ${\rm T}_z{\rm N}_z$ and then composing with
$\alpha$.)
Thus $G_x$ acts isometrically on
${\rm N}_x$, for every $x\in M$. By construction, the embeddings
${\rm N}_y\to {\rm N}_x$ are isometries.
Let then ${\rm N}'_x$ be
a linear slice at $x$ defined by using some local $G$-invariant Riemannian
metric. The map $\lambda\colon {\rm N}'_x\to {\rm N}_x$ induces a smooth
(real analytic) $G_x$-invariant Riemannian metric 
$\lambda_\ast(\alpha\vert{\rm N}_x)$ on ${\rm N}'_x$. Thus
also $\lambda$ is an isometry.
Consequently, the Riemannian metrics obtained on the linear slices satisfy the conditions
of Definition \ref{riemmetriikka}. Therefore, $\alpha$ induces a smooth
(real analytic) Riemannian metric on $M/G$. We have proved:

\begin{theorem}
\label{indmetr}
Let $G$ be a Lie group and let $M$ be a smooth (real analytic) manifold
on which $G$ acts by a proper, smooth (real analytic) almost free
action. Then the quotient $M/G$ is a smooth (real analytic)
orbifold. Every smooth (real analytic) $G$-invariant
Riemannian metric $\alpha$ on $M$  induces a smooth
(real analytic) Riemannian metric $\hat{\alpha}$ on $M/G$.
\end{theorem}

We leave it for the reader to verify the following observation:

\begin{lemma}
\label{viimeinen}
Let $G$ be a Lie group and let $M$ be a smooth (real analytic) manifold
on which $G$ acts by a proper, smooth (real analytic) almost free
action. Assume $M/G$ is connected.
Let $M_0$ be a connected component of $M$, and let
$H=\{ g\in G\mid gM_0=M_0\}$. Then the following hold:
{\begin{enumerate}
\item $H$ is a closed subgroup of $G$, and it contains the
connected component $G_0$ of the identity element of $G$,
\item the quotient orbifolds $M/G$ and $M_0/H$ are 
canonically smoothly
(real analytically) diffeomorphic,
\item there is a one-to-one correspondence between 
smooth (real analytic) $G$-invariant Riemannian metrics on $M$
and smooth (real analytic) $H$-invariant Riemannian metrics
on $M_0$,
\item there is a one-to-one correspondence between smooth
(real analytic) $G$-invariant maps $M\to{\mathbb{R}}$ and
smooth (real analytic) $H$-invariant maps $M_0\to
{\mathbb{R}}$.
\end{enumerate}}
\end{lemma}

Let $G$, $M$ and $\alpha$ be as in Theorem \ref{indmetr}.
Let $x\in M$ and let ${\rm N}_x$ be a linear slice at $x$,
defined by using $\alpha$.
Let $\delta\colon U\to G$ be a real analytic cross section 
of the map $G\to G/G_x$, $g\mapsto gG_x$, as before Theorem
\ref{indmetr}. Let $F\colon U\times {\rm N}_{x}\to V$, be the
smooth (real analytic) diffeomorphism defined by using $\delta$,
and let ${\rm pr}\colon U\times {\rm N}_{x}\to {\rm N}_x$ denote
the projection.
Let $\gamma\colon [0,1]\to V\subset G{\rm N}_x$ be a curve.
The map ${\rm pr}\circ F^{-1}$ takes every point
in $V$ to a point in the same orbit.
Thus the curves ${\rm pr}\circ F^{-1}\circ\gamma$ and
$\gamma$ induce the same curve
$[0,1]\to M/G$.  Assume there is $c\in (0,1)$ such that
$\gamma(c)=x$. Let $\gamma_0$ be the geodesic segment
connecting $({\rm pr}\circ F^{-1}\circ \gamma)(0)$ to $x$ and let $\gamma_1$ 
be the geodesic segment connecting $x$ to $({\rm pr}\circ F^{-1}\circ\gamma)(1)$. 
Then  the two
geodesic segments are contained in ${\rm N}_x$ and they intersect orthogonally
the $G$-orbits they meet  (see \cite{AKLM}, the proof of Proposition
3.1 (2)). Let $\gamma^\ast$ denote the curve $\gamma_0\cup\gamma_1$.
We obtain:

\begin{lemma}
\label{huom}
For every curve $\gamma\colon [0,1]\to V$  
such that $\gamma(c)=x$, for some $c\in (0,1)$, there is
a curve $\gamma^\ast\colon [0,1]\to {\rm N}_x$ having the
following properties:
{\begin{enumerate}
\item $L_{\hat{\alpha}}(\gamma^\ast)
=L_\alpha(\gamma^\ast)
\leq L_\alpha(\gamma)$,
\item $\pi(\gamma^\ast(0))=\pi(\gamma(0))$ 
and $\pi(\gamma^\ast(1))=\pi(\gamma(1))$.
\end{enumerate}}
\end{lemma}

We point out that for any curve $\gamma$ in
${\rm N}_x$, $L_{\hat{\alpha}}(
\gamma)$ denotes the length of $\gamma$
calculated by using the Riemannian metric $\alpha\vert
{\rm N}_x$ defined before Theorem \ref{indmetr}, while
$L_\alpha(\gamma)$ denotes the length of $\gamma$
calculated by using the  $G_x$-invariant submanifold 
Riemannian metric $\alpha$ induces on ${\rm N}_x$.
If $\gamma$ intersects orthogonally every orbit it meets, then
the two lengths are the same.

Assume $M/G$ is connected. Let $M_0$ be a connected
component of $M$, and let $H$ be the subgroup of $G$
consisting of the elements that map $M_0$ to itself, as in
Lemma \ref{viimeinen}. Let $\alpha$ be a smooth (real analytic)
$G$-invariant Riemannian metric on $M$. 
By restriction, we may consider $\alpha$
as an $H$-invariant Riemannian metric on $M_0$.
Let $d_{\alpha}$ be the $H$-invariant metric 
induced on $M_0$ by $\alpha$. The
metric $d_{\alpha}$ then induces a metric  $\widetilde{d}_{{\alpha}}$ on
$M_0/H\cong M/G$, where 
$$
\widetilde{d}_{{\alpha}}(\pi(x),\pi(y))=\inf \{ d_\alpha(x,hy)\mid h\in H\}.
$$
Let $d_{\hat\alpha}$ be the metric that 
the Riemannian metric $\hat\alpha$ induces on $M/G$. 
We will use Lemma \ref{huom} to prove the following result:

\begin{theorem}
\label{samametriikka}
Let $M$, $G$, $\alpha$ and $\hat{\alpha}$ be as in Theorem \ref{indmetr}.
Assume $M/G$ is connected.
Then $\widetilde{d}_{{\alpha}}=d_{\hat\alpha}$.
\end{theorem}

\begin{proof} By Lemma \ref{viimeinen}, we may without
loss of generality assume that
$M$ is connected.
Let $x, y\in M$. We will show that $\widetilde{d}_{{\alpha}}
(\pi(x),\pi(y))=d_{\hat\alpha}(\pi(x),\pi(y))$.  Let $\gamma\colon
[0,1]\to M/G$ be a curve such that $\gamma(0)=\pi(x)$ and
$\gamma(1)=\pi(y)$. We may assume that $\gamma$ is admissible. 
Let $\widetilde{\gamma}\colon [0,1]\to M$ be a lift of $\gamma$.
Decompose the interval $[0,1]$ into finitely many subintervals
$[t_i,t_{i+1}]$, $1\leq i\leq m$, such that  $\widetilde{\gamma}( [t_i, t_{i+1}])$
 is contained in a small neighbourhood $V_i\cong
 U_i\times {\rm N}_{x_i}$ of
 ${\rm N}_{x_i}$, as before Lemma \ref{huom}, where
${\rm N}_{x_i}$ is a linear slice at $x_i\in \widetilde{\gamma}([t_i,t_{i+1}])$.
We may assume that $x_1=\widetilde{\gamma}(0)$ and
$x_m=\widetilde{\gamma}(1)$.
For every $1<i<m$, let $c_i\in (t_i,t_{i+1})$ be such that
$\widetilde{\gamma}(c_i)=x_i$. Let
$\widetilde{\gamma}_i$ denote the restriction of $\widetilde{\gamma}$
to $[t_i,t_{i+1}]$,  for every $i$. 
By Lemma \ref{huom}, we may replace every curve
$\widetilde{\gamma}_i$ by a curve $\widetilde{\gamma}^\ast_i\colon
[t_i,t_{i+1}]\to {\rm N}_{x_i}$ having the properties that
$\pi(\widetilde{\gamma}^\ast_i(t_i))=\pi(\widetilde{\gamma}_i(t_i))$, 
$\pi(\widetilde{\gamma}^\ast_i(t_{i+1}))=\pi(\widetilde{\gamma}_i(t_{i+1}))$ and
$$
L_{\hat{\alpha}}(\widetilde{\gamma}^\ast_i)
=L_\alpha(\widetilde{\gamma}^\ast_i)
\leq L_\alpha(\widetilde{\gamma}_i).
$$
We next show  that the   $\widetilde{\gamma}^\ast _i$ 
can be chosen in such a way that they define a
curve $\widetilde{\gamma}^\ast\colon [0,1]\to M$, 
where $\pi(\widetilde{\gamma}^\ast(0))=\pi(x)$ and
$\pi(\widetilde{\gamma}^\ast(1))=\pi(y)$.
For example,
$\pi(\widetilde{\gamma}^\ast_1(t_2))=\pi(\widetilde{\gamma}^\ast_2(t_2))$,
$\widetilde{\gamma}^\ast_1(t_2)\in {\rm N}_{x_1}$ and
$\widetilde{\gamma}^\ast_2(t_2)=g\widetilde{\gamma}^\ast_1(t_2)$, for
some $g\in G$. Thus, if $\widetilde{\gamma}^\ast_1(t_2)\not=
\widetilde{\gamma}^\ast_2(t_2)$, we can replace $\widetilde{\gamma}^\ast_2$
by  $g^{-1}\circ \widetilde{\gamma}^\ast_2$.
Continuing like this,
we can replace every  $\widetilde{\gamma}^\ast_i$, if necessary, 
in such a way that we obtain a curve $\widetilde{\gamma}^\ast\colon
[0,1]\to M$. The curve $\widetilde{\gamma}^\ast$ induces  a curve
$\gamma^\ast\colon [0,1]\to M/G$ with $\gamma^\ast(0)=\pi(x)$ and
$\gamma^\ast(1)=\pi(y)$. It follows from the way $\gamma^\ast$ was
constructed  that
$$
L_{\hat\alpha}(\gamma^\ast)=
L_{\alpha}(\widetilde{\gamma}^\ast)
\leq L_{\alpha}(\widetilde{\gamma}).
$$
Since $\gamma$ was an arbitrary path from
$\pi(x)$ to $\pi(y)$, it follows that
$$
d_{\hat\alpha}(\pi(x),\pi(y))\leq
\widetilde{d}_\alpha(\pi(x),\pi(y)).
$$

Let  then $z\in M$ and let
${\rm N}_z$ be a linear slice at $z$. Let $\mu\colon [0,1]
\to {\rm N}_z$ be a curve. We may assume that $\mu$
is simple, starts at  $z$ and intersects each orbit at most once.
If $\mu([0,1])$ is orthogonal to every
orbit it meets, then $L_{\alpha}(\mu)=L_{\hat{\alpha}}(\mu)$. 
If $\mu([0,1])$ is not  orthogonal to every orbit it meets,
then we may replace $\mu$ by a curve
 $\mu^\ast\colon [0,1]
\to G{\rm N}_z$ with
$\mu^\ast(t)\in G\mu(t)$, for every $t\in [0,1]$, such that
$\mu^\ast([0,1])$ is orthogonal to every orbit it meets. 
Then 
$$
L_{\hat{\alpha}}(\mu)=
L_{{\alpha}}(\mu^\ast)\geq \widetilde{d}_{\alpha}(\pi(\mu(0)),\pi(\mu(1))).
$$
Replacing local lifts of any path from $\pi(x)$ to $\pi(y)$ in this manner
and gluing them at the endpoints shows that
$$
\widetilde{d}_\alpha(\pi(x),\pi(y))
\leq
d_{\hat\alpha}(\pi(x),\pi(y)).
$$

\end{proof}

According to Lemma 2.4 in 
\cite{Ka},  the metric $\widetilde{d}_{\alpha}$ on $M/G\cong M_0/H$
is complete if and only if  the $H$-invariant metric $d_{{\alpha}}$ 
on $M_0$ is complete.
Since, by Theorem \ref{samametriikka}, $\widetilde{d}_{{\alpha}}=d_{\hat\alpha}$, it
follows that $d_{\hat{\alpha}}$ is complete if and only if $d_{\alpha}$ is complete.
We conclude:

\begin{cor}
\label{korollaari}
Let $G$ be a Lie group and let $M$ be a smooth (real analytic) manifold
on which $G$ acts by a proper, smooth (real analytic) almost free
action. Let $\alpha$ be a $G$-invariant smooth (real analytic)
Riemannian metric on $M$ and let $\hat{\alpha}$ be the smooth
(real analytic) Riemannian metric that $\alpha$ induces on $M/G$.
Then $\hat{\alpha}$ is complete if and only if $\alpha$ is complete.
\end{cor}

\section{Comparing Riemannian metrics on $X$ and $X_{\rm red}$}
\label{compX}

\noindent Let $X$ be a smooth (real analytic) orbifold. 
Assume $X$ is not reduced. Replacing
every orbifold chart $(\widetilde{U}, G,\varphi)$ by a chart
$(\widetilde{U}, G/{\rm ker}(G), \varphi)$ yields a smooth (real analytic)
reduced orbifold $X_{\rm red}$. The orbifolds $X$ and $X_{\rm red}$
are identical as topological spaces and
the identity map $X\to X_{\rm red}$ is an orbifold map. 
Let $(\widetilde{U}, G, \varphi)$ be an orbifold chart
of $X$. Then a Riemannian metric on
$\widetilde{U}$ is invariant
under the action of $G$ if and only if it is invariant under the
action of $G/{\rm ker}(G)$. 
The following proposition follows immediately from Definition 
\ref{riemmetriikka}:

\begin{prop}
\label{one-to-one}
There is a one-to-one correspondence between Riemannian metrics on
$X$ and Riemannian metrics on $X_{\rm red}$. A Riemannian
metric $\alpha$ on $X$ is smooth (real analytic) if and only if the
corresponding Riemannian metric $\alpha_{\rm red}$ on $X_{\rm red}$ is
smooth (real analytic).
\end{prop}

\begin{remark}
\label{huomautus} 
Assume $X$ is connected.
Let $d_\alpha$ and $d_{\alpha_{\rm  red}}$ be the metrics induced on
$X$ by $\alpha$ and on $X_{\rm red}$ by $\alpha_{\rm red}$, respectively.
If we just consider $X$ and $X_{\rm red}$ as topological spaces, i.e.,
if we identify $X_{\rm red}$ with $X$, then both $d_\alpha$ and $d_{\alpha_{\rm red}}$ 
are metrics on $X$ and $d_\alpha=d_{\alpha_{\rm red}}$.  In particular,
this implies that $\alpha$ is complete 
if and only if $\alpha_{\rm red}$ is complete.
\end{remark}

\section{Real analytic Riemannian metric}
\label{real}

\noindent In this section we show that every real analytic orbifold has a
real analytic Riemannian metric. In order to do that, we first need to construct the
{\it frame bundle} ${\rm Fr}(X)$ of a reduced real analytic orbifold $X$.
The construction is similar to that in the smooth case.
For details, see \cite{MM}, pp. 42--43.

Recall that, for an $n$-dimensional real analytic manifold, the frame bundle
${\rm Fr}(M)$ is a real analytic fibre bundle over $M$, the fibre of $x\in M$
is the manifold of all ordered bases of the tangent space ${\rm T}_x(M)$.
The frame bundle ${\rm Fr}(M)$ admits a canonical right action of the general
linear group ${\rm GL}_n({\mathbb{R}})$ which makes it a principal
${\rm GL}_n({\mathbb{R}})$-bundle over $M$.

For a reduced $n$-dimensional real analytic orbifold $X$, we first form the frame bundles
${\rm Fr}(\widetilde{U}_i)$ corresponding to orbifold charts $(\widetilde{U}_i,
G_i,\varphi_i)$. The action of $G_i$ on $\widetilde{U}_i$ induces a left
action on ${\rm Fr}(\widetilde{U}_i)$:
$$
G_i\times {\rm Fr}(\widetilde{U}_i)\to {\rm Fr}(\widetilde{U}_i),\,\,\,
(g,(x, B_x))\mapsto (gx,(dg)_x(B_x)).
$$
Since $G_i$ acts effectively on $\widetilde{U}_i$, it follows that
the action of $G_i$ on ${\rm Fr}(\widetilde{U}_i)$ is free. The group
${\rm GL}_n({\mathbb{R}})$ acts on ${\rm Fr}(\widetilde{U}_i)$ from the
right and the action commutes with the action of $G_i$.  Thus
${\rm Fr}(\widetilde{U}_i)/G_i$ is a real analytic manifold on which
${\rm GL}_n({\mathbb{R}})$ acts real analytically.
In fact, we can consider ${\rm Fr}(\widetilde{U}_i)/G_i$ as a twisted product
${\widetilde{U}}_i\times_{G_i}{\rm GL}_n({\mathbb{R}})$. It now follows from
Lemma 0.1 in \cite{IK}, that ${\rm GL}_n({\mathbb{R}})$ acts
properly on ${\rm Fr}(\widetilde{U}_i)/G_i$.

Assume $A\in {\rm GL}_n({\mathbb{R}})$ and
$[x,I]A=[x,I]$. Then
$(x, A)=(gx,(dg)_x)$, for some $g\in G_i$. Thus
$g\in (G_i)_x$ and $A=(dg)_x$. It follows that the isotropy subgroups
of the ${\rm GL}_n({\mathbb{R}})$-action are finite, i.e., 
${\rm GL}_n({\mathbb{R}})$ acts almost freely on ${\rm Fr}(\widetilde{U}_i)/G_i$.

The frame bundle ${\rm Fr}(X)$ of $X$ can  be constructed by
gluing together the quotients
${\rm Fr}(\widetilde{U}_i)/G_i$. 
This is done by using the gluing maps induced by the embeddings 
$\lambda_{ij}\colon \widetilde{U}_i\to\widetilde{U}_j$ between orbifold charts.
We obtain:

\begin{theorem}
\label{apuframe}
Let $X$ be a reduced $n$-dimensional real analytic orbifold. Then the frame
bundle ${\rm Fr}(X)$ of $X$ is a real analytic manifold on which ${\rm GL}_n({\mathbb{R}})$
acts by a proper, real analytic, 
effective, almost free action. 
The orbifolds $X$ and ${\rm Fr}(X)/{\rm GL}_n({\mathbb{R}})$ are 
real analytically diffeomorphic.
\end{theorem}

We are now ready to prove Theorems \ref{realanalcase}
and \ref{Orthonframe}:

\medskip

\noindent {\it Proof of Theorem \ref{realanalcase}.}
Let us first assume that $X$ is a reduced  $n$-dimensional 
real analytic orbifold. By Theorem
\ref{apuframe},  $X\cong{\rm Fr}(X)/{\rm GL}_n({\mathbb{R}})$.
Since ${\rm GL}_n({\mathbb{R}})$ acts properly and real analytically on
${\rm Fr}(X)$, it follows from Theorem I in \cite{IK}, that ${\rm Fr}(X)$
has a real analytic ${\rm GL}_n({\mathbb{R}})$-invariant Riemannian
metric ${\alpha}$. But then, by Theorem \ref{indmetr},  ${\alpha}$ induces  a real
analytic Riemannian metric on $X$.

Let then $X$ be any real analytic orbifold, and let $X_{\rm red}$ be the
corresponding reduced orbifold. By the first part of the proof, we know that
$X_{\rm red}$ has a real analytic Riemannian metric. 
It now follows from
Proposition \ref{one-to-one}, that also $X$ has a real analytic Riemannian metric.
\qed

\medskip

\noindent{\it Proof of Theorem  \ref{Orthonframe}.}
Let $X$ be a reduced $n$-dimensional real analytic orbifold. Since 
$X$ has a real analytic Riemannian metric, by Theorem \ref{realanalcase},
we can  construct the orthonormal
frame bundle ${\rm OFr}(X)$ of $X$ (denoted by ${\rm Fr}(X)$ in \cite{ALR}), 
exactly as in the smooth case,
see pp. 11 - 12 in \cite{ALR}. The proof is now
similar to the proof of the smooth case (\cite{ALR}, Theorem 1.23).
\qed

\medskip

The following result is well-known, see Proposition 2.1 in
\cite{SU} for the smooth case. The proof of the real analytic case is
similar.

\begin{prop}
\label{seuraus}
Let $X$ be a reduced  $n$-dimensional smooth (real analytic) orbifold
and let ${\rm OFr}(X)$ be the orthonormal frame bundle of $X$. Let
$\beta$ be a smooth (real analytic) Riemannian metric on
${\rm OFr}(X)/{\rm O}(n)$. Then there is an ${\rm O}(n)$-invariant smooth
(real analytic) Riemannian metric $\alpha$ on ${\rm OFr}(X)$
such that $\beta$ equals the Riemannian metric $\hat{\alpha}$
induced on ${\rm OFr}(X)/{\rm O}(n)$ by $\alpha$.
\end{prop}

Theorem \ref{Orthonframe}, Corollary \ref{korollaari}, Proposition \ref{seuraus},
Proposition \ref{one-to-one} and the remark after it
imply the following correspondence:

\begin{cor}
\label{vastaavuus}
Let $X$ be an $n$-dimensional smooth (real analytic) orbifold, and let
$X_{\rm red}$ be the reduced orbifold corresponding to $X$. Then
every smooth (real analytic) Riemannian metric on $X$ is induced by an
${\rm O}(n)$-invariant smooth (real analytic) Riemannian metric on
${\rm OFr}(X_{\rm red})$. Conversely, any ${\rm O}(n)$-invariant
smooth (real analytic) Riemannian metric on ${\rm OFr}(X_{\rm red})$
induces a smooth (real analytic) Riemannian metric on $X$.  A
Riemannian metric on $X$ is complete if and only if it is induced
by a complete ${\rm O}(n)$-invariant Riemannian metric
on ${\rm OFr}(X_{\rm red})$.
\end{cor}

\section{Complete Riemannian metric}
\label{complete} 

\noindent Recall that two smooth (real analytic) Riemannian metrics 
$\alpha_1$ and $\alpha_2$ on a smooth 
(real analytic) orbifold
$X$ are called {\it conformal}, if there exists a 
smooth (real analytic) orbifold map
$\omega\colon X\to {\mathbb{R}}$ such that $\omega(x)>0$ for every
$x\in X$ and $\alpha_1=\omega\alpha_2$.

\medskip

\noindent {\it Proof of Theorem \ref{complthm}.} 
Let ${\rm id}\colon X\to X_{\rm red}$ be the identity map. By 
Theorem 1.23 in \cite{ALR} and
Theorem \ref{Orthonframe}, there is a smooth (real analytic)
diffeomorphism $f\colon X_{\rm red}\to {\rm OFr}(X_{\rm red})
/{\rm O}(n)$.  Let $\pi\colon {\rm OFr}(X_{\rm red})\to
{\rm OFr}(X_{\rm red}) /{\rm O}(n)$ denote the natural projection.
Let $\alpha$ be a smooth (real analytic)
Riemannian metric on $X$, and let $\alpha_{\rm red}$ be the
corresponding Riemannian metric on $X_{\rm red}$. The
diffeomorphism $f$ induces a smooth (real analytic)
Riemannian metric  $f^\ast\alpha_{\rm red}$ on ${\rm OFr}(X_{\rm red})
/{\rm O}(n)$. By Proposition \ref{seuraus}, there is an
${\rm O}(n)$-invariant smooth (real analytic) Riemannian metric
$\beta$ on ${\rm OFr}(X_{\rm red})$ such that the
Riemannian metric $\hat{\beta}$ induced on ${\rm OFr}(X_{\rm red})
/{\rm O}(n)$ by $\beta$ equals $f^\ast\alpha_{\rm red}$. By Theorems 3.1 and
5.2 in \cite{Ka}, there is an ${\rm O}(n)$-invariant smooth
(real analytic) map $\omega\colon {\rm OFr}(X_{\rm red})\to
{\mathbb{R}}$ such that the Riemannian metric $\omega^2\beta$
on ${\rm OFr}(X_{\rm red})$ is complete.
Let $\bar{\omega}\colon {\rm OFr}(X_{\rm red})/{\rm O}(n)\to
{\mathbb{R}}$ denote the map induced by $\omega$.
Then $(\bar{\omega}^2\circ f\circ {\rm id})\alpha$ is a complete
smooth (real analytic) Riemannian metric on $X$ conformal to
$\alpha$. \qed

\medskip

A Riemannian metric  $\alpha$ on a connected orbifold
$X$ is called {\it bounded} if 
$X$ is bounded with respect to the metric induced by $\alpha$.
The following result concerning bounded Riemannian metrics was
originally proved by Nomizu and Ozeki in the manifold setting
(\cite{NO}, Theorem 2).

\begin{theorem}
\label{boundthm}
Let $X$ be a connected
smooth (real analytic)  orbifold and let $\alpha$ be a smooth
(real analytic) Riemannian metric on $X$. 
Then there is a bounded smooth (real analytic) Riemannian
metric on $X$ which is  conformal to $\alpha$.
\end{theorem}

\begin{proof}  We use the same notation as in the proof of Theorem
\ref{complthm}.  By Theorem \ref{complthm}, we may assume that
$\alpha$ is complete.
Let $x_0$  be an arbitrary
point in ${\rm OFr}(X_{\rm red})$ and let
${\rm OFr}(X_{\rm red})_0$ denote the connected component of
${\rm OFr}(X_{\rm red})$ containing $x_0$.
Let $H=\{ h\in {\rm O}(n)\mid h({\rm OFr}(X_{\rm red})_0)
={\rm OFr}(X_{\rm red})_0\}$. (In fact,
$H={\rm O}(n)$, or $H={\rm SO}(n)$.)
Let $\beta$ be the
${\rm O}(n)$-invariant smooth (real analytic) Riemannian metric
on ${\rm OFr}(X_{\rm red})$ such that the
Riemannian metric $\hat{\beta}$ induced on ${\rm OFr}(X_{\rm red})
/{\rm O}(n)$ by $\beta$ equals $f^\ast\alpha_{\rm red}$, and
let $\beta_0$ denote the restriction of $\beta$ to
${\rm OFr}(X_{\rm red})_0$.
Let $d_{\beta_0}$ denote the $H$-invariant
metric $\beta_0$ induces on
${\rm OFr}(X_{\rm red})_0$. 
Let
$$
r_0\colon {\rm OFr}(X_{\rm red})_0\to {\mathbb{R}},\,\,\, x
\mapsto \max \{ d_{\beta_0}(hx_0,x)\mid h\in H\}.
$$
Then $r_0$ is a continuous $H$-invariant map and
$r_0(x)\geq d_{\beta_0}(x_0,x)$, for all $x\in {\rm OFr}(X_{\rm red})_0$. 
By Lemmas
2.3 and 5.1 in \cite{Ka}, 
there is an $H$-invariant smooth
(real analytic) map $r\colon {\rm OFr}(X_{\rm red})_0\to
{\mathbb{R}}$ such that $r(x)>r_0(x)$, for all
$x\in {\rm OFr}(X_{\rm red})_0$. 
The Riemannian metric $e^{-2r}\beta_0$ on ${\rm OFr}(X_{\rm red})_0$
is $H$-invariant and, 
by the proof of Theorem 2 in
\cite{NO}, it is bounded. 
Let $\bar{r}\colon {\rm OFr}(X_{\rm red})/{\rm O}(n)
\cong {\rm OFr}(X_{\rm red})_0/H\to
{\mathbb{R}}$ denote the map induced by $r$. Then
$e^{-2\bar{r}\circ f\circ{\rm id}}\alpha$ is a bounded
smooth (real analytic) Riemannian metric on $X$
and it is conformal to $\alpha$.
\end{proof}

Assume every Riemannian metric on $X$ is complete.  According
to Theorem \ref{boundthm}, $X$ has a bounded
complete Riemannian metric. Thus it
follows   that $X$ must be compact.

\end {document}